\documentclass[10pt, conference]{IEEEtran}
\usepackage[margin=1in,nohead,nofoot]{geometry}
\usepackage{subfigure}
\usepackage{epsfig,graphicx,psfrag}
\usepackage{amsfonts,amsmath,amssymb}
\usepackage{tikz}

\usetikzlibrary{arrows,positioning,fit}
\usetikzlibrary{calc}

\tikzset{arw/.style={->,>=stealth'}}

\newtheorem{theorem}{Theorem}[section]
\newtheorem{lemma}[theorem]{Lemma}

\newtheorem{example}[theorem]{Example}

\begin{document}

\title{A Lattice of Gambles}

\author{
\authorblockN{Paul Cuff}
\authorblockA{Princeton University}
\and
\authorblockN{Thomas Cover}
\authorblockA{Stanford University}
\and
\authorblockN{Gowtham Kumar}
\authorblockA{Stanford University}
\and
\authorblockN{Lei Zhao}
\authorblockA{Stanford University}
}

\maketitle

\begin{abstract}
A gambler walks into a hypothetical fair casino with a very real dollar bill, but by the time he leaves he's exchanged the dollar for a random amount of money.  What is lost in the process?  It may be that the gambler walks out at the end of the day, after a roller-coaster ride of winning and losing, with his dollar still intact, or maybe even with two dollars.  But what the gambler loses the moment he places his first bet is position.  He exchanges one distribution of money for a distribution of lesser quality, from which he cannot return.  Our first discussion in this work connects known results of economic inequality and majorization to the probability theory of gambling and Martingales.  We provide a simple proof that fair gambles cannot increase the Lorenz curve, and we also constructively demonstrate that any sequence of non-increasing Lorenz curves corresponds to at least one Martingale.

We next consider the efficiency of gambles.  If all fair gambles are available then one can move down the lattice of distributions defined by the Lorenz ordering.  However, the step from one distribution to the next is not unique.  Is there a sense of efficiency with which one can move down the Lorenz stream?  One approach would be to minimize the average total volume of money placed on the table.  In this case, it turns out that implementing part of the strategy using private randomness can help reduce the need for the casino's randomness, resulting in less money on the table that the casino cannot get its hands on.
\end{abstract}


\section{Introduction}
What is lost when one gambles?  Even in a hypothetical fair casino with fair bets offered, a gambler trades one random variable of wealth for another---an exchange that cannot be reversed through subsequent fair gambles.  A gambler who walks into a fair casino with \$1 walks out with a random amount of wealth with mean still \$1.  The gambler cannot devise a strategy that allows him to recover his \$1 with probability one.  He can only use fair gambles to exchange his random variable for one that is further degraded.

In this paper we identify the Lorenz curve \cite{lorenz05} as a simple way to characterize which probability distributions can be transformed through fair gambles to which others---an idea that can be distilled from the literature through an investigation of topics such as majorization, Schur convexity, second-order stochastic dominance, and related ideas in \cite{strassen65}-\cite{dubbins-savage76}.  The partial ordering of non-negative, mean-one distributions, which places $p_1 \succeq p_2$ if $p_1$ can be the starting point of a fair gambling system with $p_2$ the ending point, coincides with the partial ordering induced by the associated Lorenz curves.  Thus, given two distributions $p_a$ and $p_b$, we can use the Lorenz curve to identify the most degraded distribution from which we can gamble to either $p_a$ or $p_b$.  Similarly, we can find the least degraded distribution that can be produced from either starting point $p_a$ or $p_b$.

\begin{example}
\label{example lattice}

Consider two different mean-one distributions $p_a$ and $p_b$ given by

\begin{eqnarray*}
p_a(x) & = & \left\{
\begin{array}{ll}
1/3, & x = 0, \\
2/3, & x = 3/2,
\end{array}
\right. \\
p_b(x) & = & \left\{
\begin{array}{ll}
2/3 , & x = 1/2, \\
1/3, & x = 2.
\end{array}
\right.
\end{eqnarray*}
What is the first distribution that can be arrived at by starting from either $p_a$ or $p_b$ using fair gambles?  It's not the distribution that places all of the mass on 0 and 2.  The least degraded distribution that they can both arrive at is
\begin{eqnarray*}
p_{a \cap b}(x) & = & \left\{
\begin{array}{ll}
1/3, & x = 0, \\
1/3, & x = 1, \\
1/3, & x = 2.
\end{array}
\right. \\
\end{eqnarray*}

This is identified with ease from the Lorenz curves, illustrated in Fig. \ref{figure example lorenz}.

\begin{figure}[h]
\centering
\begin{tikzpicture}[scale=2.5]
 \draw[<->,>=stealth'] (-.1,0) node[above =5mm, left]{$p_a$} to (2.1,0);
 \draw[arw,very thick] (0,0) node[below] {$0$} -- (0,1/6);
 \draw[arw,very thick] (3/2,0) node[below] {$1.5$} -- (3/2,2/6);

 \draw[<->,>=stealth'] (-.1,-2/3) node[above =5mm, left]{$p_b$} to (2.1,-2/3);
 \draw[arw,very thick] (1/2,-2/3) node[below] {$0.5$} -- (1/2,2/6-2/3);
 \draw[arw,very thick] (2,-2/3) node[below] {$2$} -- (2,1/6-2/3);

 \draw[dotted] (-.5,-1) -- (2.4,-1);

 \draw[<->,>=stealth'] (-.1,-1.5) node[above =5mm, left]{$p_{a \cap b}$} to (2.1,-1.5);
 \draw[arw,very thick] (0,-1.5) node[below] {$0$} -- (0,1/6-1.5);
 \draw[arw,very thick] (1,-1.5) node[below] {$1$} -- (1,1/6-1.5);
 \draw[arw,very thick] (2,-1.5) node[below] {$2$} -- (2,1/6-1.5);
\end{tikzpicture}
\caption{Distributions $p_a$ and $p_b$, represented in the top two graphs, each have a mean of one, however no system of fair gambles can begin at one of these distributions and arrive at the other.  The least degraded distribution that can be reached from either $p_a$ or $p_b$ is represented by $p_{a \cap b}$ in the bottom graph.}
\label{figure example pmf}
\end{figure}
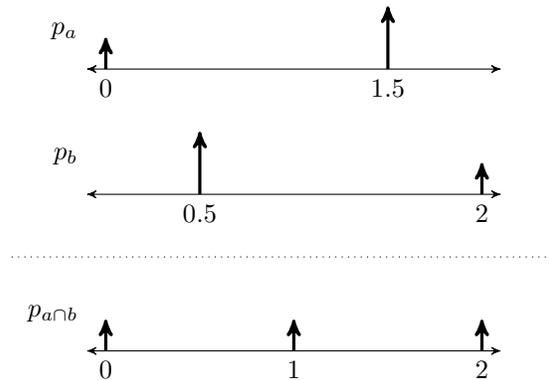

\end{example}

In this paper we show the relationship between Lorenz curves and fair gambling and discuss the lattice formed by this relationship.  This implies a ``second law'' of Martingales, characterized by the Lorenz curve or any Schur convex function, similar to entropy increase in Markov chains.  These relationships also provide a simple proof of weak convergence of Martingales.

We also consider the efficiency of going from one distribution to another.  For any system of gambles, there is an amount of money that may be lost and therefore must be placed on the table at a casino.  To model the fact that casinos actually take a profit by offering games with a house edge, we can imagine that any money placed on the table is taxed. There are in general many ways of gambling from $p_1$ to $p_2$, assuming $p_1$ is less degraded than $p_2$.  We identify a general method for constructing the most efficient gambling system.

\section{Lorenz Curve}

\subsection{Definition}

The Lorenz curve is a metric developed by Lorenz in 1905 for the purpose of measuring economic inequality in a population. If the population is ordered according to increasing wealth, forming an increasing sequence of individual wealths $w_1,w_2,...,w_N$, then the Lorenz curve is obtained from the sequence of partial sums of this sequence, $l_i = \sum_{k=1}^i w_k$. Normally, the Lorenz curve $L(u)$ is scaled so that the domain and range are $[0,1]$.  That is, $L(u)$ is the linear interpolation of the points $\{(i/N, \frac{l_i}{l_N})\}_{i=0}^N$.

In Figure \ref{figure us lorenz} we see the Lorenz curve for the United States of America (based on income rather than wealth).  We can interpret the curve in the following way.  Consider the point $A = (.6,.27)$.  This says that the poorest 60\% of the population account for 27\% of the total income.

\begin{figure}
\centering
\begin{tikzpicture}[scale=2,
 dot/.style={draw=blue,fill=blue,circle,minimum size=1mm,inner sep=0pt}]
 \draw (0,0) to node[above=10mm,midway,sloped] {Percent of Income} (0,2);
 \draw (-.02,2*.2) node[left]{20\%} -- (.02,2*.2);
 \draw (-.02,2*.4) node[left]{40\%} -- (.02,2*.4);
 \draw (-.02,2*.6) node[left]{60\%} -- (.02,2*.6);
 \draw (-.02,2*.8) node[left]{80\%} -- (.02,2*.8);
 \draw (-.02,2) node[left]{100\%} -- (3,2);
 \draw (0,0) to node[below=6mm,midway,sloped] {Percent of Households} (3,0);
 \draw (3*.2,-.02) node[below]{20\%} -- (3*.2,.02);
 \draw (3*.4,-.02) node[below]{40\%} -- (3*.4,.02);
 \draw (3*.6,-.02) node[below]{60\%} -- (3*.6,.02);
 \draw (3*.8,-.02) node[below]{80\%} -- (3*.8,.02);
 \draw (3,-.02) node[below]{100\%} -- (3,2);
 \draw[thick] (0,0) -- (3*.2,2*.04) -- (3*.4,2*.14) -- (3*.6,2*.27) node[dot,black,very thick] {} node[below right] {A} -- (3*.8,2*.53) -- (3*.95,2*.8) -- (3,2);
\end{tikzpicture}
\caption{Lorenz curve for 2003 U.S. household income distribution.}
\label{figure us lorenz}
\end{figure}
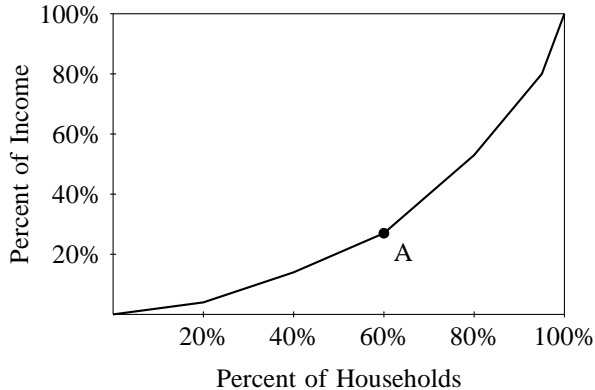

The Lorenz curve can also be defined for a mean-one probability distribution.  The interpretation is that each quantum of probability corresponds to an individual in the economic wealth distribution setting.

For a continuous random variable $X \sim p_X(x)$ we define
\begin{eqnarray*}
L(u) & \triangleq & \min_{{\cal A} \subset \Re \; : \; \mathbf{P}({\cal A}) = u} \mathbf{E} \; X \mathbf{1}_{X \in {\cal A}}.
\end{eqnarray*}

We can define the Lorenz curve in a way that accommodates all random variables (continuous, discrete, and singular) by expressing it as an optimization over all random variables $U$ that are correlated with $X$.  The general expression for the Lorenz curve is given by
\begin{eqnarray}
\label{equation lorenz general def}
L(u) & \triangleq & \min_{p_{U|X}, {\cal A} \; : \; \mathbf{P}(U \in {\cal A}) = u} \mathbf{E} \; X \mathbf{1}_{U \in {\cal A}}.
\end{eqnarray}

For example, we might let $U$ be a random variable with a uniform distribution on the interval $[0,1]$, and let $g(u)$ be a function such that $X = g(U) \sim p_X$.  This is in fact an optimal choice of $U$.  An example of a function that would work would be the inverse of the cumulative distribution function which we define by
\begin{eqnarray*}
F^{-1} (u) & \triangleq & \inf \{ x \; : \; F(x) \geq u \}.
\end{eqnarray*}
A straightforward way of calculating the Lorenz curve is
\begin{eqnarray*}
L(u) & = & \int_0^u F^{-1}(\tau) d \tau.
\end{eqnarray*}

\subsection{Properties}

\begin{enumerate}
\item $L(0) = 0$.
\item $L(1) = 1$.
\item $L(u)$ is a continuous, convex function.
\item $0 \leq L(u) \leq u, \quad \forall u \in [0,1]$.
\end{enumerate}

Any function with these properties is a Lorenz curve and uniquely specifies a mean-one distribution.

It is possible to talk about Lorenz curves for random variables that take negative values, in which case Property 4 would not necessarily hold.  We only discuss non-negative random variables in this work.

\section{Martingale Lattice}

\subsection{Feasibility}

Suppose one wishes to exchange a random variable $X_1 \sim p_{X_1}$ for another random variable $X_2 \sim p_{X_2}$ using a system of fair gambles.  The system may specify a different fair gamble, or sequence of gambles, for each value of $X_1$.  The following theorem states when such a system exists.

\begin{theorem}
\label{theorem Lorenz upstream}
Given marginal distributions $p_{X_1}$ and $p_{X_2}$ with mean one, there exists a joint distribution $p_{X_1,X_2}$ with the property that $\mathbf{E} \; (X_2|X_1) = X_1$ if and only if $L_1(u) \geq L_2(u) \; \forall u \in [0,1]$, where $L_1$ is the Lorenz curve for $p_{X_1}$ and $L_2$ is the Lorenz curve for $p_{X_2}$.
\end{theorem}

\begin{proof}
This theorem can be distilled from the literature.  The concept of majorization, which applies to vectors, is related to the Lorenz curve.  If a random variable $X$ is formed from taking the uniform distribution over the elements of a vector $x$, and a random variable $Y$ and vector $y$ share the same relationship, then $x$ majorizes $y$ if and only if the Lorenz curve for $Y$ is above the Lorenz curve for $X$.  Based on this connection, the theorem can be established.  However, we present a simple proof of the theorem here.

For the converse part of the proof, assume that a joint distribution $p_{X_1,X_2}$ has the property that $\mathbf{E} \; (X_2|X_1) = X_1$.  Construct $X_1 = g(U) \sim p_{X_1}$ where $U$ is uniformly distributed on $[0,1]$, as is used in the paragraph following \eqref{equation lorenz general def}.  Also let $U - X_1 - X_2$ form a Markov chain.

For all $u \in [0,1]$,
\begin{eqnarray*}
L_1(u) & = & \min_{{\cal A} \; : \; \mathbf{P}(U \in {\cal A}) = u} \mathbf{E} \; X_1 \mathbf{1}_{U \in {\cal A}} \\
& = & \mathbf{E} \; X_1 \mathbf{1}_{U \in {\cal A}^*} \\
& = & \mathbf{E} \; \mathbf{E} \left( X_2 | X_1 \right) \mathbf{1}_{U \in {\cal A}^*} \\
& = & \mathbf{E} \; \mathbf{E} \left( X_2 | X_1,U \right) \mathbf{1}_{U \in {\cal A}^*} \\
& = & \mathbf{E} \; \mathbf{E} \left( X_2  \mathbf{1}_{U \in {\cal A}^*} | X_1,U \right) \\
& = & \mathbf{E} \; X_2  \mathbf{1}_{U \in {\cal A}^*} \\
& \geq & \min_{p_{U|X_2}, {\cal A} \; : \; \mathbf{P}(U \in {\cal A}) = u} \mathbf{E} \; X_2 \mathbf{1}_{U \in {\cal A}} \\
& = & L_2(u).
\end{eqnarray*}
The symbol ${\cal A}^*$ refers to the $argmin$ of the right-hand side of the first equality.  The fourth equality is due to the construction of $U$ having the Markov relationship $U - X_1 - X_2$.

For the direct part of the proof, first note that $E[X_2|X_1]=X_1$ is equivalent to the existence of a sequence of fair gambles that starts with wealth $X_1$ and ends up with wealth $X_2$. For simplicity, we provide a proof for discrete random variables with finite alphabets in this paper.  The proof for random variables with general distribution will be provided in the full paper under preparation.

The following lemma solves the case where $X_1=1$, which by scaling, holds for $X_1=c$ for any constant $c>0$.

\begin{lemma}\label{lemma one point}
Let $X$ be a discrete random variable with support set $\{x_1,..,x_n\}$, $x_1< x_2...<x_n,$ and probability mass function $\{p_1,...,p_n\}$, satisfying $x_1\geq 0$ and $EX=c$. There exists a sequence of fair binary gambles that starts with \$$c$ and ends up with wealth $X$.
\end{lemma}

\begin{proof}
We give a proof by induction.  For $n=2$, the lemma automatically holds. Suppose the lemma holds for $n=k$. Consider the case $n=k+1$. First, use a fair binary gamble to generate a binary random variable $V$:
\begin{equation*}
        V=\left\{
        \begin{array}{c}
        \frac{1-p_1x_1}{1-p_1},~\text{ w.p. }{1- p_1} \\
        x_1,~\text{ w.p. }{p_1}
        \end{array}
         \right.
\end{equation*}
If $V=x_1$, we stop the gambling; if $V=\frac{1-p_1x_1}{1-p_1}$, using the inductive assumption, there exists a sequence of fair binary gambles that  starts with $\frac{1-p_1x_1}{1-p_1}$ and ends up with a random variable $X'\sim (\frac{p_2}{1-p_1},\frac{p_3}{1-p_1},..., \frac{p_n}{1-p_1} )$ with support set $\{x_2,...,x_n\}$. This completes the proof of Lemma \ref{lemma one point}.
\end{proof}

\begin{figure}[h]{
\psfrag{a1}[][][.8]{$p_X(x_1)$}
\psfrag{a2}[][][.8]{$p_X(x_2)$}
\psfrag{us}[][][1]{$(u^*,l^*)$}
\centerline{\includegraphics[width=3.5in]{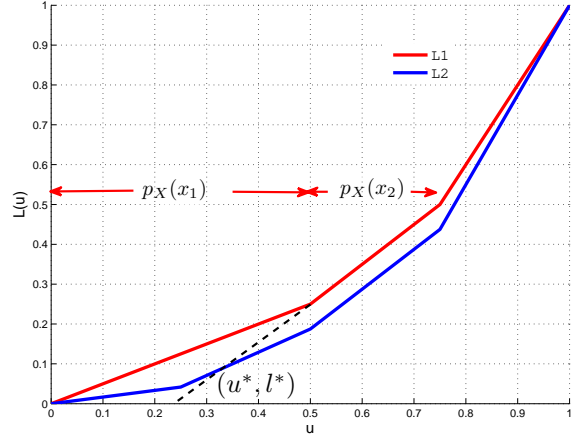}}}
\caption{Generate $p_Y$ from an upstream distribution $p_X$ using a sequence of fair binary gambles.}
\label{fig.direct_proof}
\end{figure}

Suppose $X$ and $Y$ are two non-negative discrete random variables with finite alphabets $\{x_1,...,x_n\}$ and $\{y_1,...,y_{n'}\}$, respectively, and mean one, and that the associated Lorenz curves of $X$ and $Y$ are $L_1$ and $L_2$, respectively, where $L_1(u)\geq L_2(u)$. Again, we will use induction to prove the result. Note that Lemma~\ref{lemma one point} proves the case $|\mathcal{X}|=1$, which serves as the starting case. Assume the direct part of the theorem holds for $|\mathcal{X}|=k$. Consider the case $|\mathcal{X}|=k+1$. Extend the line segment of $L_1(u)$ for $u\in [p_X(x_1),p_X(x_1)+p_X(x_2)]$ to the left, i.e., the region $u\leq p_X(x_1)$, until it intersects $L_2$ at some point $(u^*,l^*)$ The convexity of the Lorenz curve and the fact that $L_1(u)\geq L_2(u)$ guarantee that such a point exists and $u^*\leq p_X(x_1)$, illustrated in Fig.~\ref{fig.direct_proof}. Let $m^* = \mathrm{argmax}\{j: \Pr(Y< y_j) \leq u^*\}$, and define the distribution of $V$ by

\begin{equation*}
        V=\left\{
        \begin{array}{ll}
        y_i,&\text{ w.p. }\quad {\frac{p_Y(y_i)}{p_X(x_1)}, 1\leq i\leq m^*-1} \\
        y_{m^*},&\text{ w.p. }\quad \frac{u^*-\sum_{i=1}^{m^*-1}p_Y(y_i)}{p_X(x_1)}\\
        x_2,&\text{ w.p. }\quad  \frac{p_X(x_1)-u^*}{p_X(x_1)}
        \end{array}
         \right.
\end{equation*}

It is easy to check that $EV = x_1$. In the first step of the gamble, if $X\neq x_1$, we do nothing,  if $X=x_1$, we use $x_1$ to generate the random variable $V$, which is feasible due to Lemma~\ref{lemma one point}.  By doing so, we exchange $X_1$ for a random variable $Y_1$ whose Lorenz curve coincides with $L_2$ for $u \in [0,u^*]$ and coincides with $L_1$ for $u\in [p_X(x_1),1]$. Note that $Y_1|Y_1\geq x_2$ has a support set of cardinality $|\mathcal{X}|-1$, which is equal to $k$. Thus by the assumption of the induction, we can use $Y_1|Y_1\geq x_2$ to generate $L_2(u)$, $u\in (u^*,1]$. This completes the direct part of the theorem for discrete random variables with finite alphabets.
\end{proof}

\subsection{Lattice}

Theorem \ref{theorem Lorenz upstream} gives a partial ordering of mean-one distributions, based on the Lorenz curve.  We say that $p_1 \succeq p_2$, or $p_1$ is upstream from $p_2$, or $p_2$ is more degraded than $p_1$ if $L_1 \geq L_2 \; \forall u \in [0,1]$.

The partial ordering of distributions according to the Lorenz curve gives a lattice.  For any two mean-one distributions that are incomparable (the Lorenz curves cross), we can identify the last distribution (most degraded) from which they each can be produced through a system of fair gambles.  This will correspond to the point-wise maximum of the two Lorenz curves.  We can also identify the first distribution (least degraded) that they can each produce.  This is the greatest Lorenz curve below both.  Due to the convexity of the Lorenz curve, this will be the lower boundary of the convex hull of the epigraphs of both Lorenz curves.

\begin{figure}[h]{
\psfrag{La}[][][1]{$L_{a}$}
\psfrag{Lb}[][][1]{$L_{b}$}
\psfrag{Lab}[][][1]{$\quad\quad L_{p_a\bigcap p_b}$}
\psfrag{Lab2}[][][1]{$\quad\quad L_{p_a\bigcup p_b}$}
\centerline{\includegraphics[width=3.5in]{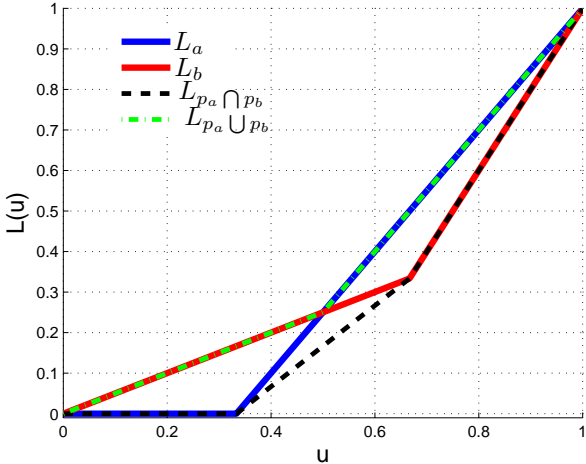}}}
\caption{This graph depicts the Lorenz curves of the least degraded distribution $L_{p_a\bigcap p_b}$ below $p_a$ and $p_a$ from Example \ref{example lattice} and the most degraded distribution $L_{p_a\bigcup p_b}$ above $p_a$ and $p_b$.}
\label{figure example lorenz}
\end{figure}

\subsection{Martingale Convergence}

Let $X_0, X_1, X_2,...$ be a sequence of random variables produced by fair gambles (wealths $X_n$ at times $n=0,1,2,...$).  Let $F_0, F_1,...$ be the corresponding sequence of cumulative distribution functions and $L_0(u), L_1(u),...$ their Lorenz functions.  Then we have the following simple proof of the convergence of $X_n$ in distribution.

\begin{theorem}
For a sequence $\{X_i\}$ resulting from fair gambles, there exists a random variable $X$ such that
\begin{eqnarray*}
X_n & \rightarrow & X \quad \mbox{in distribution.}
\end{eqnarray*}
(i.e. $F_n(x) \rightarrow F(x)$ at every point of continuity of $F$.)
\end{theorem}

\begin{proof}
The sequence of Lorenz curves $L_n(u)$ is nonnegative, monotonically nonincreasing and thus has a limit $L(u)$.  $F(x)$ is the corresponding limiting cumulative distribution function.  In the case where $L_n(u)$ converges to zero for all $u \in [0,1)$, $L(u)$ is not a Lorenz curve for lack of continuity, and the limiting distribution is zero with probability one.
\end{proof}

\subsection{Second Law of Martingales}

The second law of thermodynamics has a close relationship to the behavior of entropy in Markov chains.  For a Markov chain with a uniform stationary distribution, entropy always increases.  However, this is not true for Martingales.

Theorem \ref{theorem Lorenz upstream} indicates that Martingales degrade in a different sense.  The Lorenz curve cannot increase.  A number of measures of inequality derived from the Lorenz curve, such as the Gini index, can only increase as a Martingale progresses.

\section{Efficiently gambling down the Lorenz stream}

\subsection{Air gambles}
    Suppose the casino offers binary fair bets:
    \begin{equation*}
        X\rightarrow \left\{
        \begin{array}{c}
        2X,~\text{w.p. }{1\over 2}, \\
        ~~~0,~\text{w.p. }{1\over 2}.
        \end{array}
         \right.
    \end{equation*}
    Suppose a gambler starts with \$1 and wants to achieve a uniform distribution $\text{unif}(0,2)$
    on his wealth after gambling. There is more than one possible gambling strategy to achieve this: \\
    {\bf Method 1:}
    \begin{itemize}
        \item Bet \$${1\over 2}$, then bet \$${1\over 4}$, then \$${1\over 8}$, \dots.
        \item After infinite bets the distribution of wealth is $\text{unif}(0,2)$.
    \end{itemize}
    In this method, the gambler must place on the casino's table a total of:
    \begin{equation*}
        {1\over 2}+{1\over 4}+{1\over 8}+\dots=1.
    \end{equation*}
    {\bf Method 2:}
    \begin{itemize}
        \item Generate $Y\sim\text{unif}(0,1)$.
        \item Bet \$Y.
    \end{itemize}
     In this method, the gambler used ``air gambles'' to randomize the amount of money placed on the table and supplement the ``physical gambles'' implemented by the casino, resulting in a reduced expected amount of money placed on the table,
    \begin{equation*}
        E(Y)={1\over 2}<1.
    \end{equation*}

\subsection{Optimal efficiency}

We now ask ``What is the minimum volume of bets to be placed on the table to achieve a desired target distribution?''  We characterize the optimal efficiency for certain special cases of the target distribution.

\begin{theorem}
Suppose a casino offers all fair binary gambles, i.e. $\forall c\in (0,1)$, the casino offers
    \begin{equation*}
        c\rightarrow \left\{
        \begin{array}{ll}
            1,&\text{w.p. }{c}, \\
            0,&\text{w.p. }{1-c},
        \end{array}
         \right.
    \end{equation*}
which can be scaled up or down (both cost and payout) as desired.  Then, the minimum expected amount that one has to place on the casino's table in order to gamble from $x$ to $0$ or $1$ is given by $V^*(x)=-(1-x)\ln(1-x)\leq x$.
\end{theorem}

\begin{proof}\\
{\bf Achievability:}\\
Suppose we start with $x$. Here is a strategy to gamble to 0 or 1:
\begin{enumerate}
    \item Choose $n$ large enough and let $\delta x=\frac{x_0}{n}$.
    \item The algorithm is $n$ steps long. Let $x_n$ denote the wealth before the $n$th step.  If $x_n \neq 1$, make the following gamble:
        \begin{equation*}
            \delta x\rightarrow \left\{
            \begin{array}{ccc}
                1-x_n+\delta x,&\text{w.p. } & {\frac{\delta x}{1-x_n+\delta x}} \\
                0,&\text{w.p. } & {1-{\frac{\delta x}{1-x_n+\delta x}}}
            \end{array}
            \right.
        \end{equation*}
    \item If the gamble succeeds, we are left with $1$. If it fails, we are left with $(x-\delta x)$ and we continue the process.
\end{enumerate}

Let $V(x)$ denote the total volume of bets placed on the table to reach 0 or 1 from $x$.  Then,
\begin{equation*}
 V(x) = \delta x + \left[1-{\frac{\delta x}{1-x+\delta x}}\right] V(x-\delta x).
\end{equation*}
In the limit as $\delta x \rightarrow 0$, this can be recast as a differential equation
\begin {equation*}
\frac{dV(x)}{dx} = 1 - \frac{V(x)}{1-x}.
\end{equation*}
Solving this differential equation with boundary conditions $V(0)=V(1)=0$  gives
\begin{equation*}
V(x)=-(1-x)\ln (1-x).
\end{equation*}\\
{\bf Converse:}\\
We only outline the proof of the converse.
In the above achievability scheme, we used only ``physical'' gambles. We did not use ``air'' gambles. So we must first reason why air gambles are not needed. Note that the target distribution
    \begin{equation*}
        X\sim \left\{
        \begin{array}{ccc}
            1,&\text{w.p. }&{x} \\
            0,&\text{w.p. }&{1-x}
        \end{array}
         \right.
    \end{equation*}
is not a convex combination of two or more different distributions with mean $1$. In other words, it is an ``extreme'' distribution: an extreme point of the convex set of probability distributions with mean $1$.
A sequence of ``air'' gambles is simply a convex combination of ``physical'' gambles. Since the target distribution $X$ considered above is an extreme distribution, every physical gambling tree has to achieve $X$ on termination. The volume of bets in an ``air'' gambling scheme is the average of the volume of bets in physical gambling schemes that constitute the air gambling scheme. It follows that we need only minimize the volume $V^*(x)$ over physical gambling schemes.

Now to establish optimality over physical gambles, we use a dynamic programming approach similar to \cite{dubbins-savage76}.  According to the Bellman equation, we only need to show that $V^*$ satisfies
\begin{align*}
V^*(a+(b-a)\theta) &\leq (b-a)\theta \; + \\
&(1-\theta)V^*(a)+\theta V^*(b)
\end{align*}
whenever $0\leq a\leq b\leq1,~\text{and }0\leq\theta\leq1$.  This can be proven using some algebra and calculus.

The left hand side of the above inequality simply uses the scheme and expends a volume $V*$. The right hand side deviates from $V^*$ for only the first time step by playing an arbitrary gamble $(b-a)\theta\rightarrow(b-a)$. The inequality suggests deviating from $V^*$ is worse. By invoking this dynamic programming argument, we conclude that $V^*$ is optimal.
\end{proof}


\begin{thebibliography}{1}
\bibitem{lorenz05}
Lorenz. ``Methods of measuring the concentration of wealth.'' {\em Publications of the American Statistical Association}, Vol. 9, No. 70, 1905.

\bibitem{strassen65}
Strassen. ``The Existence of Probability Measures with Given Marginals.'' {\em The Annals of Mathematical Statistics}, Vol. 36, No. 2, 1965.

\bibitem{foster-vohra96}
Foster and Vohra. ``Calibrated Learning and Correlated Equilibrium.'' {\em Games and Economic Behavior}, Vol. 21, Issues 1-2, 1997.

\bibitem{lavenda06}
Lavenda. ``Entropies of Mixing (OEM) and the Lorenz Order.'' {\em Open Systems \& Information Dynamics}, Vol. 13, 2006.

\bibitem{arnold-vollasenor91}
Arnold and Villase\~{n}or. ``Lorenz Ordering of Order Statistics.'' {\em Lecture Notes-Monograph Series}, Vol. 19, 1991.

\bibitem{kramer05}
Kramer. ``On the Ordering of Probability Forecasts.'' {\em Sankhya:  The Indian Journal of Statistics}, Vol. 67, No. 4, 2005.

\bibitem{degroot-fienberg}
De Groot and Fienberg. ``The Comparison and Evaluation of Forecasters.'' {\em J. Royal Stat. Soc. B}, Vol. 32, 1983.

\bibitem{gneiting07}
Gneiting, Balabdaoui, and Raftery. ``Probabilitistic Forecasts, Calibration and Sharpness.'' {\em J. Royal Stat. Soc. B}, 2007.

\bibitem{kochar06}
Kochar. ``Lorenz Ordering of Order Statistics.'' {\em Statistics \& Probability Letters}, Vol. 76, 2006.

\bibitem{marshall-olkin79}
Marshall and Olkin. ``Inequalities:  Theory of Majorization and Its Applications.'' {\em Academic Press}, 1979.

\bibitem{gini12}
Gini. ``Variability and Mutability, Contribution to the Study of Statistical Distribution and Relations.'' {\em Studi Economico-Giuricici della R, Universita de Cagliari}, 1912.

\bibitem{aaberge01}
Aaberge. ``Axiomatic Characterization of the Gini Coefficient and Lorenz Curve Orderings.'' {\em Journal of Economic Theory}, Vol. 101, 2001.

\bibitem{dubbins-savage76}
Dubbins and Savage. ``Inequalities for Stochastic Processes; How to Gamble If You Must.'' {\em Dover Publications}, 1976.
\end{thebibliography}
\end{document}